\documentclass{amsart}

\usepackage{amssymb,amsmath,amsthm,latexsym,booktabs}

\theoremstyle{definition}
\newtheorem{definition}{Definition}

\newtheorem{algorithm}[definition]{Algorithm}
\newtheorem{example}[definition]{Example}

\theoremstyle{plain}
\newtheorem{lemma}[definition]{Lemma}
\newtheorem{proposition}[definition]{Proposition}
\newtheorem{theorem}[definition]{Theorem}
\newtheorem{corollary}[definition]{Corollary}
\newtheorem{conjecture}[definition]{Conjecture}

\begin{document}

\title{Leibniz triple systems}

\author{Murray R. Bremner and Juana S\'anchez-Ortega}

\address{Department of Mathematics and Statistics, University of Saskatchewan, Canada}

\email{bremner@math.usask.ca}

\address{Department of Algebra, Geometry and Topology, University of M\'alaga, Spain}

\email{jsanchez@agt.cie.uma.es}

\begin{abstract}
We define Leibniz triple systems in a functorial manner using the algorithm of Kolesnikov and Pozhidaev
which converts identities for algebras into identities for dialgebras.
We verify that Leibniz triple systems are the natural analogues of Lie triple systems in the context of dialgebras 
by showing that both the iterated bracket in a Leibniz algebra and the permuted associator in a Jordan dialgebra
satisfy the defining identities for Leibniz triple systems. 
We construct the universal Leibniz envelopes of Leibniz triple systems and prove that every identity satisfied
by the iterated bracket in a Leibniz algebra is a consequence of the defining identities for Leibniz triple systems.
To conclude, we present some examples of 2-dimensional Leibniz triple systems and their universal Leibniz envelopes.
\end{abstract}

\maketitle

%%%%%%%%%%%%%%%%%%%%%%%%%%%%%%%%%%%%%%%%%%%%%%%%%%%%%%%%%%%%%%%%%%%%%%%%

\section{Introduction}

In this paper we introduce Leibniz triple systems, which are related to Leibniz algebras
in the same way that Lie triple systems are related to Lie algebras. 
Our motivation is to present a new type of ternary algebra with potential applications in theoretical physics.
See our recent paper \cite{BSO} for a related result on the partially alternating ternary sum 
in an associative dialgebra.

We start by recalling in Section \ref{sectionpreliminaries} the definitions of associative dialgebras and Leibniz algebras.
We state the Kolesnikov-Pozhidaev (KP) algorithm which takes as input the defining identities for a variety of 
algebras and produces as output the defining identities for the corresponding variety of dialgebras.
As examples, we recall how associative dialgebras and Leibniz algebras can be obtained from associative and Lie algebras 
by an application of this algorithm.

In Section \ref{sectionleibniztriplesystems} we apply the KP algorithm to Lie triple systems, and
obtain a new variety of triple systems; we call these structures \emph{Leibniz triple systems}.
We show that this variety of structures may be characterized by two multilinear identities.
In Section \ref{sectionoperatoridentities} we reformulate the defining identities for Leibniz triple systems
in terms of left and right multiplication operators.

In Section \ref{sectionfreeleibnizalgebras} we use the structure theory for free Leibniz algebras to verify that
any subspace of a Leibniz algebra closed under the iterated Leibniz bracket is a Leibniz triple system.
In Section \ref{sectionjordandialgebras} we prove that any subspace of a Jordan dialgebra (quasi-Jordan algebra)
closed under the associator is a Leibniz triple system.

In Section \ref{sectionuniversalleibnizenvelopes} we construct universal Leibniz envelopes for Leibniz triple systems.
From this we obtain the corollary that every polynomial identity satisfied by the iterated bracket in 
a Leibniz algebra is a consequence of the defining identities for Leibniz triple systems.

In Section \ref{sectiontwodimensional} we conclude the paper with a conjectured classification of 2-di\-men\-sional 
Leibniz triple systems; we also construct their universal Leibniz envelopes.

In summary, our results demonstrate that Leibniz triple systems are the natural analogue of Lie
triple systems in the context of dialgebras.

%%%%%%%%%%%%%%%%%%%%%%%%%%%%%%%%%%%%%%%%%%%%%%%%%%%%%%%%%%%%%%%%%%%%%%%%

\section{Preliminaries} \label{sectionpreliminaries}

\subsection{Dialgebras and Leibniz algebras}

Dialgebras were introduced by Loday \cite{LodayDialgebras} (see also \cite{LodaySurvey}) 
to provide a natural setting for Leibniz algebras, a ``noncommutative'' version of Lie algebras.

\begin{definition} (Cuvier \cite{Cuvier}, Loday \cite{LodayLeibniz})
A \textbf{(right) Leibniz algebra} is a vector space $L$, together with a bilinear map 
$L \times L \to L$, denoted $(a,b) \mapsto \langle a,b \rangle$, satisfying the 
\textbf{(right) Leibniz identity}, which says that right multiplications are derivations:
  \begin{equation} \label{Rightleibnizidentity}
  \langle \langle a, b \rangle, c\rangle 
  \equiv 
  \langle \langle a, c \rangle, b \rangle 
  + 
  \langle a, \langle b, c \rangle \rangle.
  \end{equation}
If $\langle a, a \rangle \equiv 0$ then the Leibniz identity is the Jacobi identity and $L$
is a Lie algebra.
\end{definition}

An associative algebra becomes a Lie algebra if the product $ab$ is replaced by the Lie bracket $ab - ba$. 
If we replace the product $ab$ and its opposite $ba$ by two distinct operations 
$a \dashv b$ and $b \vdash a$, 
the we obtain the notion of an associative dialgebra, in which the Leibniz bracket
$a \dashv b - b \vdash a$ is not necessarily skew-symmetric.

\begin{definition} \label{definitiondialgebra}
An \textbf{associative dialgebra} is a vector space $A$ with two bilinear maps $A \times A \to A$, 
denoted $\dashv$ and $\vdash$ and called the \textbf{left} and \textbf{right} products, 
satisfying the \textbf{left and right bar identities}, and \textbf{left, right and inner associativity}:
\begin{alignat*}{3}
&
( a \dashv b ) \vdash c
\equiv
( a \vdash b ) \vdash c,
&\quad
&
a \dashv ( b \dashv c )
\equiv
a \dashv ( b \vdash c ),
\\
&
( a \dashv b ) \dashv c
\equiv
a \dashv ( b \dashv c ),
&\quad
&
( a \vdash b ) \vdash c
\equiv
a \vdash ( b \vdash c ),
&\quad
&
( a \vdash b ) \dashv c
\equiv
a \vdash ( b \dashv c ).
\end{alignat*}
The Leibniz bracket in an associative dialgebra satisfies the Leibniz identity.
\end{definition}

\subsection{KP Algorithm}

Kolesnikov \cite{Kolesnikov} introduced a general categorical framework for transforming the defining identities 
of a variety of binary algebras (associative, Lie, Jordan, etc.) into the defining identities of the 
corresponding variety of dialgebras.
This procedure was extended by Pozhidaev (in an unpublished preprint) 
to varieties of arbitrary $n$-ary (multioperator) algebras.
In this subsection we  present a simplified statement of the Kolesnikov-Pozhidaev (KP) algorithm. 

\begin{algorithm}
The input is a multilinear polynomial identity of degree $d$ for an $n$-ary operation;
the output is a collection of $d$ multilinear identities of degree $d$ for $n$ new $n$-ary operations. 

Part 1:
We consider a multilinear $n$-ary operation, denoted by the symbol
  \begin{equation} \label{operation}
  \{-,\dots,-\} 
  \qquad
  \text{($n$ arguments)}.
  \end{equation}
Given a multilinear identity of degree $d$ in this operation, 
we describe the application of the algorithm to one monomial, 
and extend this by linearity to the entire identity. 
Let $\overline{a_1 a_2 \dots a_d}$
be a multilinear monomial of degree $d$, where the bar denotes some placement of $n$-ary operation symbols \eqref{operation}. 
We introduce $n$ new $n$-ary operations, using the same operation symbol but distinguished by subscripts:
  \begin{equation} \label{noperations}
  \{-,\dots,-\}_1,
  \quad
  \dots,
  \quad
  \{-,\dots,-\}_n.
  \end{equation}
For each $i = 1, \dots, d$ we convert the monomial $\overline{a_1 a_2 \dots a_d}$
in the $n$-ary operation \eqref{operation} 
into a new monomial of the same degree in the $n$ new $n$-ary operations \eqref{noperations},
according to the following rule, based on the position of the indeterminate $a_i$. 
For each occurrence of operation \eqref{operation} in the monomial, 
either $a_i$ occurs within one of the $n$ arguments or not, and we have two 
cases:
  \begin{itemize}
  \item
  If $a_i$ occurs in the $j$-th argument then we convert this occurrence of $\{\dots\}$ to 
  the $j$-th new operation symbol $\{\dots\}_j$.
  \item
  If $a_i$ does not occur in any of the $n$ arguments, then either
    \begin{itemize}
    \item
    $a_i$ occurs to the left of this occurrence of $\{\dots\}$:
    we convert $\{\dots\}$ to the first new operation symbol $\{\dots\}_1$, or
    \item
    $a_i$ occurs to the right of this occurrence of $\{\dots\}$:
    we convert $\{\dots\}$ to the last new operation symbol $\{\dots\}_n$.
    \end{itemize}
  \end{itemize}
In step $i$ we call $a_i$ the central indeterminate of the monomial.

Part 2:
We also include the following identities,
analogous to the bar identities for associative dialgebras,
for all $i, j = 1, \dots, n$ with $i \ne j$ 
and all $k, \ell = 1, \dots, n$: 
  \begin{align*}  
  &
  \{ a_1, \dots, a_{i-1}, \{ b_1, \cdots, b_n \}_k, a_{i+1}, \dots, a_n \}_j
  \equiv
  \\ 
  &
  \{ a_1, \dots, a_{i-1}, \{ b_1, \cdots, b_n \}_\ell, a_{i+1}, \dots, a_n \}_j.
  \end{align*} 
This identity says that the $n$ new operations are interchangeable in 
the $i$-th argument of the $j$-th new operation when $i \ne j$.
\end{algorithm}

\begin{example}
The defining identities for associative dialgebras can be obtained by applying
the KP algorithm to the associativity identity $( a \circ b ) \circ c \equiv a \circ ( b \circ c )$.
The operation $\circ$ produces two new operations $\circ_1$ and $\circ_2$.
Part 1 of the algorithm gives three identities of degree 3,
making $a$, $b$, $c$ in turn the central indeterminate:
  \[
  ( a \circ_1 b ) \circ_1 c \equiv a \circ_1 ( b \circ_1 c ), \;
  ( a \circ_2 b ) \circ_1 c \equiv a \circ_2 ( b \circ_1 c ), \;
  ( a \circ_2 b ) \circ_2 c \equiv a \circ_2 ( b \circ_2 c ).
  \] 
Part 2 of the algorithm gives two identities:
  \[
  a \circ_1 ( b \circ_1 c ) \equiv a \circ_1 ( b \circ_2 c ), 
  \qquad
  ( a \circ_1 b ) \circ_2 c \equiv ( a \circ_2 b ) \circ_2 c.
  \] 
If we write $a \dashv b$ for $a \circ_1 b$ and $a \vdash b$ for $a \circ_2 b$
then we obtain Definition \ref{definitiondialgebra}.
\end{example}

\begin{example}
The defining identities for Leibniz algebras (Lie dialgebras)
can be obtained by applying the KP algorithm to the defining identities for Lie algebras:
anticommutativity (in its multilinear form) and the Jacobi identity,
  \[
  [a,b] + [b,a] \equiv 0,
  \qquad
  [[a,b],c] + [[b,c],a] + [[c,a],b] \equiv 0.
  \]
Part 1 of the algorithm produces the following five identities:
  \begin{alignat*}{2}
  &
  [a,b]_1 + [b,a]_2 \equiv 0,
  &\qquad
  &
  [[a,b]_1,c]_1 + [[b,c]_2,a]_2 + [[c,a]_2,b]_1 \equiv 0,
  \\   
  &
  [a,b]_2 + [b,a]_1 \equiv 0,
  &\qquad 
  &
  [[a,b]_2,c]_1 + [[b,c]_1,a]_1 + [[c,a]_2,b]_2 \equiv 0,
  \\
  &&&
  [[a,b]_2,c]_2 + [[b,c]_2,a]_1 + [[c,a]_1,b]_1 \equiv 0.
  \end{alignat*}
The two identities of degree 2 are both equivalent to $[a,b]_2 \equiv -[b,a]_1$,
so the second operation is superfluous.
Eliminating the second operation from the three identities of degree 3 shows that each of them is 
equivalent to the identity
  \[
  [[a,b]_1,c]_1 + [a,[c,b]_1]_1 - [[a,c]_1,b]_1 \equiv 0.
  \]
If we write $\langle a, b \rangle = [a,b]_1$ then we obtain an identity equivalent to \eqref{Rightleibnizidentity}.
Part 2 of the KP algorithm produces the following two identities:
  \[
  [ a, [ b, c ]_1 ]_1 \equiv [ a, [ b, c ]_2 ]_1,
  \qquad
  [ [ a, b ]_1, c ]_2 \equiv [ [ a, b ]_2, c ]_2.
  \] 
Eliminating the second operation reduces these to right anticommutativity:
  \begin{equation} \label{rightanticommutativity}
  \langle a, \langle b, c \rangle \rangle
  +
  \langle a, \langle c, b \rangle \rangle
  \equiv
  0.
  \end{equation} 
Setting $b = c$ in \eqref{Rightleibnizidentity} gives 
$\langle a, \langle b, b \rangle \rangle \equiv 0$,
and the linearization of this is \eqref{rightanticommutativity}.
We conclude that the output of the KP algorithm is equivalent to identity \eqref{Rightleibnizidentity}.
\end{example}

%%%%%%%%%%%%%%%%%%%%%%%%%%%%%%%%%%%%%%%%%%%%%%%%%%%%%%%%%%%%%%%%%%%%%%%%

\section{Leibniz triple systems} \label{sectionleibniztriplesystems}

\begin{definition}
A \textbf{Lie triple system} is a vector space $T$ with a trilinear operation $T \times T \times T \to T$,
denoted $(a,b,c) \mapsto [a,b,c]$, satisfying these identities:
  \allowdisplaybreaks
  \begin{align}
  &[a,b,c] + [b,a,c] \equiv 0,
  \label{skew} \tag{L1}
  \\
  &[a,b,c] + [b,c,a] + [c,a,b] \equiv 0,
  \label{cyclic} \tag{L2}
  \\
  &[a,b,[c,d,e]] - [[a,b,c],d,e] - [c,[a,b,d],e] - [c,d,[a,b,e]] \equiv 0.
  \label{derivation} \tag{L3}
  \end{align}
\end{definition}

We apply Part 1 of the KP algorithm to \eqref{skew}, \eqref{cyclic}, \eqref{derivation} and obtain 11 identities:
  \allowdisplaybreaks
  \begin{align}
  &[a,b,c]_1 + [b,a,c]_2 \equiv 0,
  \label{skew1}
  \\
  &[a,b,c]_2 + [b,a,c]_1 \equiv 0,
  \label{skew2}
  \\
  &[a,b,c]_3 + [b,a,c]_3 \equiv 0,
  \label{skew3}
  \\
  &[a,b,c]_1 + [b,c,a]_3 + [c,a,b]_2 \equiv 0,
  \label{cyclic1}
  \\
  &[a,b,c]_2 + [b,c,a]_1 + [c,a,b]_3 \equiv 0,
  \label{cyclic2}
  \\
  &[a,b,c]_3 + [b,c,a]_2 + [c,a,b]_1 \equiv 0,
  \label{cyclic3}
  \\
  &[a,b,[c,d,e]_1]_1 - [[a,b,c]_1,d,e]_1 - [c,[a,b,d]_1,e]_2 - [c,d,[a,b,e]_1]_3 \equiv 0,
  \label{derivation1}
  \\
  &[a,b,[c,d,e]_1]_2 - [[a,b,c]_2,d,e]_1 - [c,[a,b,d]_2,e]_2 - [c,d,[a,b,e]_2]_3 \equiv 0,
  \label{derivation2}
  \\
  &[a,b,[c,d,e]_1]_3 - [[a,b,c]_3,d,e]_1 - [c,[a,b,d]_1,e]_1 - [c,d,[a,b,e]_1]_1 \equiv 0,
  \label{derivation3}
  \\
  &[a,b,[c,d,e]_2]_3 - [[a,b,c]_3,d,e]_2 - [c,[a,b,d]_3,e]_2 - [c,d,[a,b,e]_1]_2 \equiv 0,
  \label{derivation4}
  \\
  &[a,b,[c,d,e]_3]_3 - [[a,b,c]_3,d,e]_3 - [c,[a,b,d]_3,e]_3 - [c,d,[a,b,e]_3]_3 \equiv 0.
  \label{derivation5}
  \end{align}
Identities \eqref{skew1} and \eqref{skew2} are both equivalent to 
  \begin{equation}
  [a,b,c]_2 \equiv - [b,a,c]_1,
  \label{reduce2}
  \end{equation}
which shows that the second operation is superfluous.
Identities \eqref{cyclic1}, \eqref{cyclic2}, \eqref{cyclic3} are equivalent, 
and applying \eqref{reduce2} to eliminate the second operation we obtain
  \begin{equation}
  [a,b,c]_3 \equiv - [b,c,a]_2 - [c,a,b]_1 \equiv [c,b,a]_1 - [c,a,b]_1,
  \label{reduce3}
  \end{equation}
which shows that the third operation is superfluous. 
Applying \eqref{reduce2} to identities \eqref{derivation1}--\eqref{derivation5} 
to eliminate the second operation we obtain
  \allowdisplaybreaks
  \begin{align}
  [a, b, [c, d, e]_1]_1 - [[a, b, c]_1, d, e]_1 + [[a, b, d]_1, c, e]_1 - [c, d, [a, b, e]_1]_3 &\equiv 0, 
  \label{derivation1a}
  \\
  - [b, a, [c, d, e]_1]_1 + [[b, a, c]_1, d, e]_1 - [[b, a, d]_1, c, e]_1 + [c, d, [b, a, e]_1]_3 &\equiv 0, 
  \label{derivation2a}
  \\
  [a, b, [c, d, e]_1]_3 - [[a, b, c]_3, d, e]_1 - [c, [a, b, d]_1, e]_1 - [c, d, [a, b, e]_1]_1 &\equiv 0, 
  \label{derivation3a}
  \\
  - [a, b, [d, c, e]_1]_3 + [d, [a, b, c]_3, e]_1 + [[a, b, d]_3, c, e]_1 - [d, c, [a, b, e]_1]_1 &\equiv 0, 
  \label{derivation4a}
  \\
  [a, b, [c, d, e]_3]_3 - [[a, b, c]_3, d, e]_3 - [c, [a, b, d]_3, e]_3 - [c, d, [a, b, e]_3]_3 &\equiv 0.
  \label{derivation5a}
  \end{align}
Identities \eqref{derivation1a} and \eqref{derivation2a} are equivalent, 
as are \eqref{derivation3a} and \eqref{derivation4a}.
Applying \eqref{reduce3} to identities \eqref{derivation1a}, \eqref{derivation3a}, \eqref{derivation5a} 
to eliminate the third operation we obtain
  \allowdisplaybreaks
  \begin{align*}
  &
  [a, b, [c, d, e]_1]_1 - [[a, b, c]_1, d, e]_1 + [[a, b, d]_1, c, e]_1 - [[a, b, e]_1, d, c]_1 
  \\
  &\quad
  + [[a, b, e]_1, c, d]_1
  \equiv 0,
  \\
  &
  [[c, d, e]_1, b, a]_1 - [[c, d, e]_1, a, b]_1 - [[c, b, a]_1, d, e]_1 + [[c, a, b]_1, d, e]_1 
  \\
  &\quad
  - 
  [c, [a, b, d]_1, e]_1 - [c, d, [a, b, e]_1]_1
  \equiv 0,
  \\
  &
  [[e, d, c]_1, b, a]_1 - [[e, c, d]_1, b, a]_1 - [[e, d, c]_1, a, b]_1 + [[e, c, d]_1, a, b]_1 
  \\
  &\quad
  - 
  [e, d, [c, b, a]_1]_1 + [e, d, [c, a, b]_1]_1 + [e, [c, b, a]_1, d]_1 - [e, [c, a, b]_1, d]_1   
  \\
  &\quad
  - 
  [e, [d, b, a]_1, c]_1 + [e, [d, a, b]_1, c]_1 + [e, c, [d, b, a]_1]_1 - [e, c, [d, a, b]_1]_1 
  \\
  &\quad
  -
  [[e, b, a]_1, d, c]_1 + [[e, a, b]_1, d, c]_1 + [[e, b, a]_1, c, d]_1 - [[e, a, b]_1, c, d]_1 
  \equiv 0.
  \end{align*}
If we write $\langle a,b,c \rangle$ for $[a,b,c]_1$ then we obtain the identities
  \allowdisplaybreaks
  \begin{align}
  &
  \langle a, b, \langle c, d, e \rangle \rangle  
  - 
  \langle \langle a, b, c \rangle, d, e \rangle  
  + 
  \langle \langle a, b, d \rangle, c, e \rangle  
  - 
  \langle \langle a, b, e \rangle, d, c \rangle  
  \label{derivation1b}
  \\
  &\quad
  + 
  \langle \langle a, b, e\rangle , c, d\rangle 
  \equiv 0,
  \notag
  \\
  &
  \langle \langle c, d, e \rangle, b, a \rangle  
  - 
  \langle \langle c, d, e \rangle, a, b \rangle  
  - 
  \langle \langle c, b, a \rangle, d, e \rangle  
  + 
  \langle \langle c, a, b \rangle, d, e \rangle  
  \label{derivation3b}
  \\
  &\quad
  - 
  \langle c, \langle a, b, d \rangle, e \rangle  
  - 
  \langle c, d, \langle a, b, e \rangle \rangle 
  \equiv 0,
  \notag
  \\
  &
  \langle \langle e, d, c \rangle, b, a \rangle  
  - 
  \langle \langle e, c, d \rangle, b, a \rangle  
  - 
  \langle \langle e, d, c \rangle, a, b \rangle  
  + 
  \langle \langle e, c, d \rangle, a, b \rangle  
  \label{derivation5b}
  \\
  &\quad
  - 
  \langle e, d, \langle c, b, a \rangle \rangle  
  + 
  \langle e, d, \langle c, a, b \rangle \rangle  
  + 
  \langle e, \langle c, b, a \rangle, d \rangle  
  - 
  \langle e, \langle c, a, b \rangle, d \rangle    
  \notag
  \\
  &\quad
  - 
  \langle e, \langle d, b, a \rangle, c \rangle  
  + 
  \langle e, \langle d, a, b \rangle, c \rangle  
  + 
  \langle e, c, \langle d, b, a \rangle \rangle  
  - 
  \langle e, c, \langle d, a, b \rangle \rangle  
  \notag
  \\
  &\quad
  -
  \langle \langle e, b, a \rangle, d, c \rangle  
  + 
  \langle \langle e, a, b \rangle, d, c \rangle  
  + 
  \langle \langle e, b, a \rangle, c, d \rangle  
  - 
  \langle \langle e, a, b \rangle, c, d \rangle  
  \equiv 0.
  \notag
  \end{align}

We apply Part 2 of the KP algorithm and obtain 12 identities:
  \allowdisplaybreaks
  \begin{align}
  [ a, [ b, c, d ]_1, e ]_1 &\equiv [ a, [ b, c, d ]_2, e ]_1,
  \label{1.2.a}
  \\
  [ a, [ b, c, d ]_1, e ]_1 &\equiv [ a, [ b, c, d ]_3, e ]_1,
  \label{1.2.b}
  \\
  [ a, b, [ c, d, e ]_1 ]_1 &\equiv [ a, b, [ c, d, e ]_2 ]_1,
  \label{1.3.a}
  \\
  [ a, b, [ c, d, e ]_1 ]_1 &\equiv [ a, b, [ c, d, e ]_3 ]_1,
  \label{1.3.b}
  \\
  [ [ a, b, c ]_1, d, e ]_2 &\equiv [ [ a, b, c ]_2, d, e ]_2,
  \label{2.1.a}
  \\
  [ [ a, b, c ]_1, d, e ]_2 &\equiv [ [ a, b, c ]_3, d, e ]_2,
  \label{2.1.b}
  \\  
  [ a, b, [ c, d, e ]_1 ]_2 &\equiv [ a, b, [ c, d, e ]_2 ]_2,
  \label{2.3.a}
  \\
  [ a, b, [ c, d, e ]_1 ]_2 &\equiv [ a, b, [ c, d, e ]_3 ]_2,
  \label{2.3.b}
  \\
  [ [ a, b, c ]_1, d, e ]_3 &\equiv [ [ a, b, c ]_2, d, e ]_3,
  \label{3.1.a}
  \\
  [ [ a, b, c ]_1, d, e ]_3 &\equiv [ [ a, b, c ]_3, d, e ]_3,
  \label{3.1.b}
  \\  
  [ a, [ b, c, d ]_1, e ]_3 &\equiv [ a, [ b, c, d ]_2, e ]_3,
  \label{3.2.a}
  \\
  [ a, [ b, c, d ]_1, e ]_3 &\equiv [ a, [ b, c, d ]_3, e ]_3.
  \label{3.2.b}
  \end{align}
Applying \eqref{reduce2} and \eqref{reduce3} to identities \eqref{1.2.a}--\eqref{3.2.b}
to eliminate the second and third operations we obtain
  \allowdisplaybreaks
  \begin{align}
  &
  \langle  a,  \langle  b, c, d  \rangle , e  \rangle  
  \equiv 
  -  
  \langle  a,  \langle  c, b, d  \rangle , e  \rangle ,
  \label{1.2.aa}
  \\
  &
  \langle  a,  \langle  b, c, d  \rangle , e  \rangle  
  \equiv  
  \langle  a,  \langle  d, c, b  \rangle , e  \rangle  
  -  
  \langle  a,  \langle  d, b, c  \rangle , e  \rangle ,
  \label{1.2.bb}
  \\
  &
  \langle  a, b,  \langle  c, d, e  \rangle   \rangle  
  \equiv 
  -  
  \langle  a, b,  \langle  d, c, e  \rangle   \rangle ,
  \label{1.3.aa}
  \\
  &
  \langle  a, b,  \langle  c, d, e  \rangle   \rangle  
  \equiv  
  \langle  a, b,  \langle  e, d, c  \rangle   \rangle  
  -  
  \langle  a, b,  \langle  e, c, d  \rangle   \rangle ,
  \label{1.3.bb}
  \\
  &
  -  
  \langle  d,  \langle  a, b, c  \rangle , e  \rangle  
  \equiv  
  \langle  d,  \langle  b, a, c  \rangle , e  \rangle ,
  \label{2.1.aa}
  \\
  &
  -  
  \langle  d,  \langle  a, b, c  \rangle , e  \rangle  
  \equiv 
  -  
  \langle  d,  \langle  c, b, a  \rangle , e  \rangle  
  +  
  \langle  d,  \langle  c, a, b  \rangle , e  \rangle ,
  \label{2.1.bb}
  \\  
  &
  -  
  \langle  b, a,  \langle  c, d, e  \rangle   \rangle  
  \equiv  
  \langle  b, a,  \langle  d, c, e  \rangle   \rangle ,
  \label{2.3.aa}
  \\
  &
  -  
  \langle  b, a,  \langle  c, d, e  \rangle   \rangle  
  \equiv 
  -  
  \langle  b, a,  \langle  e, d, c  \rangle   \rangle  
  +  
  \langle  b, a,  \langle  e, c, d  \rangle   \rangle ,
  \label{2.3.bb}
  \\
  &
  \langle  e, d,  \langle  a, b, c  \rangle   \rangle  
  -  
  \langle  e,  \langle  a, b, c  \rangle , d  \rangle  
  \equiv 
  -  
  \langle  e, d,  \langle  b, a, c  \rangle   \rangle  
  +  
  \langle  e,  \langle  b, a, c  \rangle , d  \rangle ,
  \label{3.1.aa}
  \\
  &
  \langle  e, d,  \langle  a, b, c  \rangle   \rangle  
  -  
  \langle  e,  \langle  a, b, c  \rangle , d  \rangle 
  \equiv 
  \label{3.1.bb}
  \\
  &\quad
  \langle  e, d,  \langle  c, b, a  \rangle   \rangle  
  -  
  \langle  e,  \langle  c, b, a  \rangle , d  \rangle  
  -  
  \langle  e, d,  \langle  c, a, b  \rangle   \rangle  
  +  
  \langle  e,  \langle  c, a, b  \rangle , d  \rangle ,
  \notag
  \\  
  &
  \langle  e,  \langle  b, c, d  \rangle , a  \rangle  
  -  
  \langle  e, a,  \langle  b, c, d  \rangle   \rangle  
  \equiv 
  -  
  \langle  e,  \langle  c, b, d  \rangle , a  \rangle  
  +  
  \langle  e, a,  \langle  c, b, d  \rangle   \rangle ,
  \label{3.2.aa}
  \\
  &
  \langle  e,  \langle  b, c, d  \rangle , a  \rangle  
  -  
  \langle  e, a,  \langle  b, c, d  \rangle   \rangle 
  \equiv 
  \label{3.2.bb}
  \\
  &\quad
  \langle  e,  \langle  d, c, b  \rangle , a  \rangle  
  -  
  \langle  e, a,  \langle  d, c, b  \rangle   \rangle 
  -  
  \langle  e,  \langle  d, b, c  \rangle , a  \rangle  
  +  
  \langle  e, a,  \langle  d, b, c  \rangle   \rangle .
  \notag
  \end{align}
We rewrite \eqref{1.2.aa} and use it to rewrite \eqref{1.2.bb} as follows,
  \allowdisplaybreaks
  \begin{align}
  &
  \langle  a,  \langle  b, c, d  \rangle , e  \rangle  
  +  
  \langle  a,  \langle  c, b, d  \rangle , e  \rangle  
  \equiv 
  0,
  \label{2skew}
  \\
  &
  \langle  a,  \langle  b, c, d  \rangle , e  \rangle  
  +  
  \langle  a,  \langle  c, d, b  \rangle , e  \rangle  
  +  
  \langle  a,  \langle  d, b, c  \rangle , e  \rangle  
  \equiv 
  0.
  \label{2cyclic}
  \end{align}
Similarly \eqref{1.3.aa} and \eqref{1.3.bb} become,
  \allowdisplaybreaks
  \begin{align}
  &
  \langle  a, b,  \langle  c, d, e  \rangle   \rangle  
  +  
  \langle  a, b,  \langle  d, c, e  \rangle   \rangle  
  \equiv 
  0,
  \label{3skew}
  \\
  &
  \langle  a, b,  \langle  c, d, e  \rangle   \rangle  
  +  
  \langle  a, b,  \langle  d, e, c  \rangle   \rangle  
  +  
  \langle  a, b,  \langle  e, c, d  \rangle   \rangle  
  \equiv 
  0.
  \label{3cyclic}
  \end{align}
We remark that \eqref{2skew}--\eqref{3cyclic} show that the inner triple in a monomial of 
the second or third association types,
$\langle -, \langle -, -, - \rangle, - \rangle$ 
and 
$\langle -, -, \langle -, -, - \rangle \rangle$,
has properties analogous to identities \eqref{skew} and \eqref{cyclic}.
It is clear that \eqref{2.1.aa}--\eqref{2.3.bb} are equivalent to \eqref{2skew}--\eqref{3cyclic}.
We note that \eqref{3skew}--\eqref{3cyclic} are immediate consequences of \eqref{derivation1b}.
We have reduced the output of the KP algorithm to identities 
\eqref{derivation1b}--\eqref{derivation5b} and \eqref{2skew}--\eqref{2cyclic}. 
We show that \eqref{derivation5b} is redundant.
Applying \eqref{2skew}--\eqref{3cyclic} to identity \eqref{derivation5b} we obtain
  \begin{align*}
  &
  \big( \,
  \langle \langle e, d, c \rangle, b, a \rangle - \langle \langle e, d, c \rangle, a, b \rangle - \langle \langle e, b, a \rangle, d, c \rangle + \langle \langle e, a, b \rangle, d, c \rangle 
  \\
  &\quad
  - \langle e, \langle a, b, d \rangle, c \rangle - \langle e, d, \langle a, b, c \rangle \rangle 
  \, \big)
  \\
  - \,
  &
  \big( \,
  \langle \langle e, c, d \rangle, b, a \rangle - \langle \langle e, c, d \rangle, a, b \rangle 
  - \langle \langle e, b, a \rangle, c, d \rangle + \langle \langle e, a, b \rangle, c, d \rangle 
  \\
  &\quad
  - \langle e, \langle a, b, c \rangle, d \rangle - \langle e, c, \langle a, b, d \rangle \rangle
  \, \big)
  \equiv 0,
  \end{align*}
which follows from \eqref{derivation3b}. 
This completes the proof of the following result.

\begin{theorem} \label{LTS}
Applying the KP algorithm to Lie triple systems produces the variety of ternary algebras with 
a trilinear operation $\langle -,-,- \rangle$ satisfying these identities:
\allowdisplaybreaks
  \begin{align}
  &
  \langle a, \langle b, c, d \rangle, e \rangle 
  + 
  \langle a, \langle c, b, d \rangle, e \rangle 
  \equiv 
  0,
  \label{LTS1} 
  \tag{LTS1}
  \\
  &
  \langle a, \langle b, c, d \rangle, e \rangle 
  + 
  \langle a, \langle c, d, b \rangle, e \rangle 
  + 
  \langle a, \langle d, b, c \rangle, e \rangle 
  \equiv 
  0,
  \label{LTS2} 
  \tag{LTS2}
  \\
  &
  \langle a, b, \langle c, d, e \rangle \rangle  
  - 
  \langle \langle a, b, c \rangle, d, e \rangle  
  + 
  \langle \langle a, b, d \rangle, c, e \rangle  
  - 
  \langle \langle a, b, e \rangle, d, c \rangle  
  \label{LTS-B} 
  \tag{LTS-B}  
  \\
  &\quad
  + 
  \langle \langle a, b, e\rangle , c, d\rangle 
  \equiv 0,
  \notag
  \\
  &
  \langle \langle c, d, e \rangle, b, a \rangle  
  - 
  \langle \langle c, d, e \rangle, a, b \rangle  
  - 
  \langle \langle c, b, a \rangle, d, e \rangle  
  + 
  \langle \langle c, a, b \rangle, d, e \rangle  
  \label{LTS3} 
  \tag{LTS3}
  \\
  &\quad
  - 
  \langle c, \langle a, b, d \rangle, e \rangle  
  - 
  \langle c, d, \langle a, b, e \rangle \rangle 
  \equiv 0.
  \notag
  \end{align}
\end{theorem}

We remark that \eqref{LTS-B} shows that monomials in the third association type 
$\langle -,-,\langle -,-,-\rangle \rangle$ 
can be expressed as linear combinations of monomials in the first association type 
$\langle \langle -,-,-\rangle ,-,- \rangle $. 
Moreover, in the last four terms of \eqref{LTS-B}, the signs and permutations of $c,d,e$ 
correspond to the expansion of the Lie triple product $-[[c,d],e]$ in an associative algebra. 
We therefore introduce an identity analogous to \eqref{LTS-B} 
but for the second association type $\langle -,\langle -,-,-\rangle,- \rangle$:
  \begin{align}
  &
  \langle a, \langle b, c, d \rangle, e \rangle  
  - \langle \langle a, b, c \rangle, d, e \rangle  
  + \langle \langle a, c, b \rangle, d, e \rangle  
  + \langle \langle a, d, b \rangle, c, e \rangle  
  \label{LTS-A} \tag{LTS-A}
  \\
  & 
  - \langle \langle a, d, c \rangle, b, e \rangle  
  \equiv 0.
  \notag
  \end{align}
This new identity expresses monomials in the second association type 
as linear combinations of monomials in the first association type;
in the last four terms of \eqref{LTS-A}, the signs and permutations of $b,c,d$ 
correspond to $-[[b,c],d]$.

\begin{lemma} \label{2identities}
\eqref{LTS-A}, \eqref{LTS-B} are equivalent to \eqref{LTS1}, \eqref{LTS2}, \eqref{LTS-B}, \eqref{LTS3}.
\end{lemma}

\begin{proof}
We write $S_1$, $S_2$, $S_4$, $S_A$, $S_B$ for the the left sides 
of identities \eqref{LTS1}, \eqref{LTS2}, \eqref{LTS3}, \eqref{LTS-A}, \eqref{LTS-B} respectively.
The equations  
  \begin{align*}
  S_1(a,b,c,d,e) &= S_A(a,b,c,d,e) + S_A(a,c,b,d,e),
  \\
  S_2(a,b,c,d,e) &= S_A(a,b,c,d,e) + S_A(a,d,b,c,e) + S_A(a,c,d,b,e),
  \\
  S_4(a,b,c,d,e) &= {} - S_A(c,a,b,d,e)- S_B(c,d,a,b,e) ,
  \\
  S_A(a,b,c,d,e) &= S_1(a,b,c,d,e) + S_4(c,b,a,d,e) + S_B(a,d,c,b,e),    
  \end{align*}
can be verified by direct calculation.
\end{proof}

\begin{definition} \label{definitionLTS}
A \textbf{Leibniz triple system} is a vector space $T$ with a trilinear operation $T \times T \times T \to T$ 
denoted $\langle -,-,-\rangle $ satisfying \eqref{LTS-A} and \eqref{LTS-B}:
  \allowdisplaybreaks
  \begin{align*}
  \langle a,\langle b,c,d\rangle,e\rangle & \equiv 
  \langle \langle a,b,c \rangle ,d,e \rangle - \langle \langle a,c,b\rangle,d,e\rangle - \langle \langle a,d,b   \rangle,c,e\rangle + \langle \langle a,d,c\rangle,b,e\rangle,
  \\ 
  \langle a,b,\langle c,d,e\rangle\rangle & \equiv 
  \langle \langle a,b,c\rangle,d,e\rangle - \langle \langle a,b,d\rangle,c,e\rangle - \langle \langle a,b,e\rangle,c,d\rangle + \langle \langle a,b,e\rangle,d,c\rangle.
  \end{align*}
\end{definition}

We remark that Lemma \ref{2identities} shows that Leibniz triple systems can also be characterized by 
the four identities of Theorem \ref{LTS}. 
We will use without further comment the most convenient characterization for our purposes. 

\begin{example} \label{Lie2LeibnizExample}
Let $T$ be a Lie triple system with product $[-,-,-]$. 
It is easy to check that $T$ is a Leibniz triple system.
If $a,b,c,d,e \in T$ then
  \begin{align*}
  &
  [a,b,[c,d,e]] - [[a,b,c],d,e] + [[a,b,d],c,e] - [[a,b,e],d,c] + [[a,b,e],c,d] \stackrel{\eqref{derivation}}{\equiv} 
  \\
  &
  [c,[a,b,d],e] + [c,d,[a,b,e]] + [[a,b,d],c,e] - [[a,b,e],d,c] + [[a,b,e],c,d] \stackrel{\eqref{skew}}{\equiv} 
  \\
  &
  [c,d,[a,b,e]] + [[a,b,e],c,d] + [d,[a,b,e],c] \stackrel{\eqref{cyclic}}{\equiv} 0,
  \end{align*}
which proves \eqref{LTS-B}, and the proof of \eqref{LTS-A} is similar way.
Hence an associative algebra gives a Leibniz triple system if we set
$\langle a,b,c \rangle = abc - bac - cab + cba$.
\end{example}

\begin{example}
Let $A$ be a differential associative algebra in the sense of Loday \cite{LodaySurvey}:
an associative algebra $A$ with a product $a \cdot b$ 
and a linear map $d\colon A \to A$ such that $d^2 = 0$ and 
$d(a \cdot b) = d(a) \cdot b + a \cdot d(b)$ for all $a, b \in A$.
One endows $A$ with a dialgebra structure by defining 
$a \dashv b = a \cdot d(b)$ and $a \vdash b = d(a) \cdot b$. 
Then $A$ becomes a Leibniz triple system (see Corollary \ref{corollarydialgebra}) if we set
  \[
  \langle a,b,c \rangle 
  = 
  a \cdot d(b) \cdot d(c) -  d(b) \cdot a \cdot d(c) - d(c) \cdot a \cdot d(b) + d(c) \cdot d(b) \cdot a.
  \]
\end{example}

%%%%%%%%%%%%%%%%%%%%%%%%%%%%%%%%%%%%%%%%%%%%%%%%%%%%%%%%%%%%%%%%%%%%%%%%

\section{Operator identities for Leibniz triple systems} \label{sectionoperatoridentities}

In this section we present a more intuitive formulation of 
the defining identities for Leibniz triple systems.

\begin{definition}
Let $T$ be a triple system with product $\{-,-,-\}$.
For $a, b \in T$ we define two endomorphisms of $T$ as follows:
  \[
  L_{a,b}(x) =  \{a,b,x\},
  \qquad
  R_{a,b}(x) =  \{x,a,b\} - \{x,b,a\}.
  \]
Let $D$ be an endomorphism of $T$. 
We call $D$ a \textbf{ternary derivation} if it satisfies
  \[
  D\big( \{a,b,c\} \big) 
  \equiv
  \{D(a),b,c\}
  +
  \{a,D(b),c\}
  +
  \{a,b,D(c)\}.
  \]
We call $D$ a \textbf{Lie triple derivation} if it satisfies
  \[
  D\big( \{a,b,c\} \big) 
  \equiv
  \{D(a),b,c\}
  -
  \{D(b),a,c\}
  -
  \{D(c),a,b\}
  +
  \{D(c),b,a\}.
  \]
The right side has the form of the expansion of the Lie triple product $[[a,b],c]$
using the Lie bracket $[a,b] = ab - ba$; this motivates our choice of terminology.
\end{definition}

\begin{corollary}
Every Leibniz triple system satisfies the following identities:
  \begin{align}
  R_{a,b}( \langle c,d,e \rangle)
  &\equiv
  \langle R_{a,b}(c), d, e \rangle 
  + \langle c, R_{a,b}(d), e \rangle 
  + \langle c, d, R_{a,b}(e) \rangle,
  \label{Rab} \tag{OP1}
  \\
  L_{a,b}( \langle c,d,e \rangle)
  &\equiv
  \langle L_{a,b}(c), d, e \rangle 
  - \langle L_{a,b}(d), c, e \rangle 
  - \langle L_{a,b}(e),c,d \rangle 
  \label{Lab} \tag{OP2}
  \\
  &\quad
  + \langle L_{a,b}(e),d,c \rangle,
  \notag
  \\
  [ R_{a,b}, R_{c,d} ]
  &\equiv
  R_{R_{a,b}(c),d}
  -
  R_{R_{a,b}(d),c},
  \label{bracket R} \tag{OP3}
  \\
  [ R_{c,d}, L_{a,b} ]
  &\equiv
  L_{L_{a,b}(c),d}
  -
  L_{L_{a,b}(d),c}.
  \label{bracket of R and L} \tag{OP4}
  \end{align}
In particular, $R_{ab}$ is a derivation and $L_{ab}$ is a Lie triple derivation.
\end{corollary}

\begin{proof}
We rewrite \eqref{LTS3} as 
  \begin{align*}
  \langle \langle c, d, e \rangle, a, b \rangle  
  - \langle \langle c, d, e \rangle, b, a \rangle  
  &\equiv  
  \langle \langle c, a, b \rangle, d, e \rangle  
  - \langle \langle c, b, a \rangle, d, e \rangle  
  - \langle c, \langle a, b, d \rangle, e \rangle
  \\
  &\quad  
  - \langle c, d, \langle a, b, e \rangle \rangle,
\end{align*}
and note that 
  \begin{align*}
  \langle c,\langle a, b, d \rangle ,e \rangle 
  &\stackrel{\eqref{LTS2}}{\equiv} 
  - \langle c, \langle d, a, b \rangle, e \rangle 
  - \langle c, \langle b, d, a \rangle, e \rangle
  \\
  &\stackrel{\eqref{LTS1}}{\equiv} 
  - \langle c, \langle d, a, b \rangle, e \rangle 
  + \langle c, \langle d, b, a \rangle, e \rangle, 
  \\
  \langle c, d, \langle a, b, e \rangle \rangle 
  &\stackrel{\eqref{3cyclic}}{\equiv} 
  - \langle c, d, \langle e, a, b \rangle \rangle 
  - \langle c, d, \langle b, e, a \rangle \rangle 
  \\
  &\stackrel{\eqref{3skew}}{\equiv} 
  - \langle c, d, \langle e, a, b \rangle \rangle
  + \langle c, d, \langle e, b, a \rangle \rangle.
  \end{align*}
Combining these equations gives \eqref{Rab}, and \eqref{Lab} is equivalent to \eqref{LTS-A}.
Identity \eqref{bracket R} follows directly from \eqref{derivation5b}. 
We rewrite identity \eqref{LTS-B} as
  \begin{align*}
  \langle  \langle a, b, c \rangle, d, e \rangle 
  -  \langle  \langle a, b, d \rangle, c, e \rangle  
  &\equiv
  \langle  \langle a, b, e \rangle, c, d \rangle 
  -  \langle  \langle a, b, e \rangle, d, c \rangle 
  +  \langle a, b,  \langle   c, d, e \rangle \rangle. 
  \end{align*}
The left side equals
  \[
  L_{L_{a,b}(c),d}(e) - L_{L_{a,b}(d),c}(e).
  \]
Applying \eqref{3skew}--\eqref{3cyclic} to $\langle a, b, \langle c, d, e \rangle \rangle$
the right side becomes
  \[
  \langle  \langle a, b, e \rangle, c, d \rangle 
  -  \langle  \langle a, b, e \rangle, d, c \rangle 
  -  \langle a, b,  \langle e, c, d \rangle \rangle 
  + \langle a, b,  \langle e, d, c \rangle \rangle,
  \]
which equals
  \[
  ( R_{c,d} \circ L_{a,b} ) (e)
  -
  ( L_{a,b} \circ R_{c,d} ) (e)
  =
  [ R_{c,d}, L_{a,b} ] (e).
  \]
This proves \eqref{bracket of R and L}.
\end{proof}

%%%%%%%%%%%%%%%%%%%%%%%%%%%%%%%%%%%%%%%%%%%%%%%%%%%%%%%%%%%%%%%%%%%%%%%%

\section{Leibniz triple systems from Leibniz algebras} \label{sectionfreeleibnizalgebras}

In this section we recall from Loday and Pirashvili \cite{LodayPirashvili} 
(see also Loday \cite{LodaySurvey}) the structure of free Leibniz algebras.
From this we obtain a large class of examples of Leibniz triple systems.

Let $V$ be a vector space over a field $F$.
For $m \ge 1$ we consider the $m$-th tensor power $V^{\otimes m} = V \otimes_F \cdots \otimes_F V$ ($m$ factors)
which is spanned by simple tensors $v_1 \otimes \cdots \otimes v_m$.
The (non-unital) tensor algebra of $V$ is 
  \[
  A(V) 
  =
  \bigoplus_{m \ge 1} V^{\otimes m},
  \]
with associative multiplication defined on simple tensors by concatenation,
  \[
  ( v_1 \otimes \cdots \otimes v_m )
  ( v_{m+1} \otimes \cdots \otimes v_{m+n} )
  =
  v_1 \otimes \cdots \otimes v_{m+n},
  \]
and extended bilinearly.
We make $A(V)$ into a Leibniz algebra in which the product, denoted $x \cdot y$,
is defined to be the unique Leibniz product for which 
  \begin{equation} \label{leftfactordegree1}
  ( v_1 \otimes \cdots \otimes v_m ) \cdot v_{m+1}
  =
  v_1 \otimes \cdots \otimes v_m \otimes v_{m+1}.
  \end{equation}
That is, we define $x \cdot y$ inductively on the degree $n$ of $y$, 
using \eqref{leftfactordegree1} for $n = 1$ and the following equation for $n \ge 2$:
  \begin{align}
  &
  ( v_1 \otimes \cdots \otimes v_m ) \cdot ( v_{m+1} \otimes \cdots \otimes v_{m+n} )
  =
  \label{leftfactordegreen}
  \\
  &
  v_1 \otimes \cdots \otimes v_{m+n}
  -
  ( v_1 \otimes \cdots \otimes v_m \otimes v_{m+n} )
  \cdot
  ( v_{m+1} \otimes \cdots \otimes v_{m+n-1} ).
  \notag
  \end{align}
If we write
  \[
  x = v_1 \otimes \cdots \otimes v_m,
  \quad
  y = v_{m+1} \otimes \cdots \otimes v_{m+n-1},
  \quad
  z = v_{m+n}.
  \]
then \eqref{leftfactordegreen} expresses the Leibniz identity in the form
  \[
  x \cdot ( y \cdot z ) = ( x \cdot y ) \cdot z - ( x \cdot z ) \cdot y.
  \]
This inductive definition shows that simple tensors 
correspond to left-normalized Leibniz products:
  \[
  v_1 \otimes v_2 \cdots \otimes v_m
  =
  ( \text{---} \, ( ( v_1 \cdot v_2 ) \cdot v_3 ) \, \text{---} \cdot v_{m-1} ) \cdot v_m.
  \]
We may therefore omit the tensor symbols, since we only need one association type 
in each degree.
Roughly speaking, right (left) multiplication by an element of $V$ makes the left (right) factor 
an associative product (left-normalized Lie product).

\begin{example} \label{example_LB_Monomials}
For $a, b, c, d \in V$, we have
  \allowdisplaybreaks 
  \begin{align*} 
  &
  a \cdot b 
  = 
  a b,
  \\
  &
  a b \cdot c 
  = 
  a b c,
  \quad
  a \cdot b c 
  = 
  a b c - a c b,
  \\
  &
  a b c \cdot d 
  = 
  a b c d,
  \quad
  a b \cdot c d 
  =
  a b c d - a b d c,
  \quad
  a \cdot b c d
  =
  a b c d
  -
  a c b d
  -
  a d b c
  -
  a d c b.  
  \end{align*}
  \end{example}

\begin{definition}
We write $A(V)_L$ for the vector space $A(V)$ with the Leibniz product defined by equation \eqref{leftfactordegree1}. 
\end{definition}

\begin{theorem} \emph{(Loday and Pirashvili \cite{LodayPirashvili})}
The Leibniz algebra $A(V)_L$ is the free Leibniz algebra on the vector space $V$.
\end{theorem}

The iterated Lie bracket $[[-,-],-]$ in a Lie algebra satisfies
the defining identities for Lie triple systems.
An analogous result holds in the Leibniz setting.

\begin{proposition}
Any subspace of a Leibniz algebra with product $\langle -, -\rangle$ which is closed under 
the trilinear operation $\langle \langle -, - \rangle, - \rangle$ is a Leibniz triple system.
\end{proposition}

\begin{proof}
It suffices to verify that identities \eqref{LTS-A} and \eqref{LTS-B} are satisfied by the operation 
$\langle a,b,c\rangle = \langle \langle a, b\rangle, c\rangle$.
Since every Leibniz algebra is a quotient of a free Leibniz algebra, we only need to prove the claim for free
Leibniz algebras.
The existence of a normal form for basis monomials in the free Leibniz algebra implies that we can reduce the
proof to a straightforward computation.
We first note that any ternary monomial in degree 5 in the first association type is already in normal form as a monomial
in the free Leibniz algebra: 
  \[
  \langle \langle a,b,c \rangle, d,e \rangle = (((ab)c)d)e.
  \]
Applying the Leibniz identity, for \eqref{LTS-A} and \eqref{LTS-B} we have
  \begin{align*}
  \langle a, b, \langle c, d, e \rangle \rangle
  &=
  (ab)((cd)e)
  =
  ((ab)(cd))e - ((ab)e)(cd)
  \\
  &=
  (((ab)c)d)e - (((ab)d)c)e - (((ab)e)c)d + (((ab)e)d)c
  \\
  &=
  \langle \langle a,b,c \rangle ,d,e\rangle 
  - \langle \langle a,b,d\rangle,c,e\rangle 
  - \langle \langle a,b,e\rangle,c,d\rangle 
  + \langle \langle a,b,e\rangle,d,c\rangle,
  \\
  \langle a, \langle b, c, d \rangle, e \rangle
  &=
  (a((bc)d))e
  =
  ((a(bc))d)e - ((ad)(bc))e
  \\
  &=
  (((ab)c)d)e - (((ac)b)d)e - (((ad)b)c)e + (((ad)c)b)e
  \\
  &=
  \langle \langle a,b,c \rangle,d,e \rangle 
  - \langle \langle a,c,b\rangle,d,e\rangle 
  - \langle \langle a,d,b\rangle,c,e\rangle 
  + \langle \langle a,d,c\rangle,b,e\rangle.
  \end{align*}
This completes the proof.
\end{proof}

\begin{corollary} \label{corollarydialgebra}
A subspace of an associative dialgebra is a Leibniz triple system if it is closed under the trilinear operation
  \[
  \langle a, b, c \rangle
  =
  a \dashv b \dashv c - b \vdash a \dashv c - c \vdash a \dashv b + c \vdash b \vdash a.
  \]
\end{corollary}

\begin{proof}
Any such subspace is a Leibniz algebra with $\langle a, b \rangle = a \dashv b - b \vdash a$.
\end{proof}

%%%%%%%%%%%%%%%%%%%%%%%%%%%%%%%%%%%%%%%%%%%%%%%%%%%%%%%%%%%%%%%%%%%%%%%%

\section{Leibniz triple systems from Jordan dialgebras} \label{sectionjordandialgebras}

In this section we prove that the permuted associator in a Jordan dialgebra
satifies the defining identities for Leibniz triple systems.
Thus any subspace of a Jordan dialgebra which is closed under the associator
becomes a Leibniz triple system.
This generalizes the classical result that the associator in a Jordan algebra satisfies
the defining identities for Lie triple systems.

\begin{definition} \label{JDdefinition}
(Kolesnikov \cite{Kolesnikov}, Vel\'asquez and Felipe \cite{VelasquezFelipe1}, Bremner \cite{Bremner})
Over a field of characteristic not 2 or 3, a \textbf{(right) Jordan dialgebra} 
is a vector space with a bilinear operation $ab$, 
satisfying these polynomial identities:
  \begin{alignat*}{2}
  &\textbf{right commutativity:}
  &\qquad
  a(bc) &\equiv a(cb),
  \\
  &\textbf{right Jordan identity:}
  &\qquad
  ( b a^2 ) a &\equiv ( b a ) a^2,
  \\
  &\textbf{right Osborn identity:}
  &\qquad
  ( a, b, c^2 ) &\equiv 2 ( a c, b, c ).
  \end{alignat*}
(Algebras satisfying the last identity were systematically studied by Osborn \cite{Osborn}.)
\end{definition}

\begin{lemma}
The linearized forms of the right Jordan and Osborn identities are
  \begin{align}
  &
  RJ(a,b,c,d) = 
  \label{RJidentity}
  \\
  &\quad 
  (d(ab))c + (d(ac))b + (d(bc))a - (da)(bc) - (db)(ac) - (dc)(ab),
  \notag
  \\
  &
  RO(a,b,c,d) = 
  \label{ROidentity}
  \\
  &
  \quad
  ((ac)b)d + ((ad)b)c - (ab)(cd) - (ac)(bd) - (ad)(bc) + a((cd)b).
  \notag
  \end{align}
\end{lemma}

\begin{proof}
For a general discussion of linearization of polynomial identities in nonassociative algebras, 
see Chapter 1 of Zhevlakov et al.~\cite{Zhevlakov}.
\end{proof}

Following Bremner and Peresi \cite{BremnerPeresi},
we use the following order on the association types in degree 5 in a free right commutative algebra:
  \begin{alignat*}{4}
  &1\colon  (((ab)c)d)e  &\qquad
  &2\colon  ((a(bc))d)e  &\qquad
  &3\colon  ((ab)(cd))e  &\qquad
  &4\colon  (a((bc)d))e  
  \\
  &5\colon  ((ab)c)(de)  &\qquad
  &6\colon  (a(bc))(de)  &\qquad
  &7\colon  (ab)((cd)e)  &\qquad
  &8\colon  a(((bc)d)e)  
  \\
  &9\colon  a((bc)(de))
  \end{alignat*}
The consequences of right commutativity in degree 5 can be expressed by the following symmetries
of the association types:
  \begin{align*}
  &
  1\colon (((ab)c)d)e \;\; \text{has no symmetries} 
  \\
  &
  2\colon ((a(bc))d)e = ((a(cb))d)e	
  \qquad\qquad
  3\colon	((ab)(cd))e = ((ab)(dc))e 
  \\
  &
  4\colon (a((bc)d))e = (a((cb)d))e	
  \qquad\qquad
  5\colon	((ab)c)(de) = ((ab)c)(ed) 
  \\
  &
  6\colon (a(bc))(de) = (a(cb))(de) = (a(bc))(ed) 
  \\
  &
  7\colon (ab)((cd)e) = (ab)((dc)e)	
  \qquad\qquad
  8\colon	a(((bc)d)e) = a(((cb)d)e) 
  \\
  &
  9\colon a((bc)(de)) = a((cb)(de)) = a((bc)(ed)) = a((de)(bc))
  \end{align*}

\begin{theorem}
Let $L$ be a subspace of a Jordan dialgebra $J$ which is closed under the associator
$(a,b,c) = (ab)c - a(bc)$.
Then $L$ is a Leibniz triple system with the trilinear operation defined to be the permuted associator 
$\langle a, b, c \rangle = (a,c,b)$.
\end{theorem}

\begin{proof}
It suffices to verify that identities \eqref{LTS1}, \eqref{LTS2}, \eqref{LTS-B}, \eqref{LTS3} 
are satisfied by the permuted associator.
We first consider \eqref{LTS1} and \eqref{LTS2}; we show in fact that these identities follow from 
right commutativity, without using the Jordan and Osborn identities.
In \eqref{LTS1} and \eqref{LTS2} we replace each occurrence of the operation $\langle a, b, c \rangle$ 
by the permuted associator $(a,c,b)$:
  \begin{align*}
  &
  ( a, e, ( b, d, c ) )
  +
  ( a, e, ( c, d, b ) )
  \equiv 0,
  \\
  &
  ( a, e, ( b, d, c ) )
  +
  ( a, e, ( c, b, d ) )
  +
  ( a, e, ( d, c, b ) )
  \equiv 0.
  \end{align*}
Expanding the associators gives
  \begin{align*}
  &
  (ae)((bd)c) - (ae)(b(dc)) - a(e((bd)c)) + a(e(b(dc)))
  \\
  &\quad
  +
  (ae)((cd)b) - (ae)(c(db)) - a(e((cd)b)) + a(e(c(db)))
  \equiv 0,
  \\
  &
  (ae)((bd)c) - (ae)(b(dc)) - a(e((bd)c)) + a(e(b(dc)))
  \\
  &\quad
  +
  (ae)((cb)d) - (ae)(c(bd)) - a(e((cb)d)) + a(e(c(bd)))
  \\
  &\quad
  +
  (ae)((dc)b) - (ae)(d(cb)) - a(e((dc)b)) + a(e(d(cb)))
  \equiv 0.
  \end{align*}
Both equations are immediate consequences of right commutativity.  

We next consider identity \eqref{LTS-B}.
Replacing each occurrence of $\langle a, b, c \rangle$ by the permuted associator $(a,c,b)$ gives
  \[
  (a,(c,e,d),b) - ((a,c,b),e,d) + ((a,d,b),e,c) - ((a,e,b),c,d) + ((a,e,b),d,c).
  \]
Expanding the associators produces the following expression:
  \begin{align*}
  &
  (a((ce)d))b - (a(c(ed)))b - a(((ce)d)b) + a((c(ed))b)
  \\
  - \;
  &
  (((ac)b)e)d + ((a(cb))e)d + ((ac)b)(ed) - (a(cb))(ed)
  \\
  + \;
  &
  (((ad)b)e)c - ((a(db))e)c - ((ad)b)(ec) + (a(db))(ec)
  \\
  - \;
  &
  (((ae)b)c)d + ((a(eb))c)d + ((ae)b)(cd) - (a(eb))(cd)
  \\
  + \;
  &
  (((ae)b)d)c - ((a(eb))d)c - ((ae)b)(dc) + (a(eb))(dc).
  \end{align*}
We straighten each term using right commutativity, and sort the terms 
by association type and then by lex order of the permutation;
some of the terms cancel:
  \begin{equation}
  \label{LTS5part1}
  \left\{
  \begin{array}{l}
  {}
  - (((ac)b)e)d 
  + (((ad)b)e)c 
  - (((ae)b)c)d 
  + (((ae)b)d)c 
  \\[3pt]
  {}
  + ((a(bc))e)d  
  - ((a(bd))e)c 
  + ((a(be))c)d 
  - ((a(be))d)c 
  \\[3pt]
  {}
  + (a((ce)d))b 
  - (a((de)c))b 
  + ((ac)b)(de) 
  - ((ad)b)(ce) 
  \\[3pt]
  {}
  - (a(bc))(de)
  + (a(bd))(ce)
  - a(((ce)d)b) 
  + a(((de)c)b).
  \end{array}
  \right. 
  \end{equation}  
Consider the following expression, which clearly vanishes in any Jordan dialgebra:
  \begin{align*}
  &
    RJ( ce, b, d, a ) 
  - RJ( de, b, c, a ) 
  + RJ( b, c, e, a ) d 
  - RJ( b, d, e, a ) c 
  \\
  &\quad
  - RO( a, b, ce, d ) 
  + RO( a, b, de, c ) 
  - RO( a, b, c, e ) d 
  + RO( a, b, d, e ) c. 
  \end{align*} 
Expanding each term using equations \eqref{RJidentity} and \eqref{ROidentity} gives
  \allowdisplaybreaks
  \begin{align*}
  &
  (a((ce)b))d 
  + (a((ce)d))b 
  + (a(bd))(ce) 
  - (a(ce))(bd) 
  - (ab)((ce)d) 
  - (ad)((ce)b) 
  \\
  - \;
  &
  (a((de)b))c 
  - (a((de)c))b 
  - (a(bc))(de) 
  + (a(de))(bc) 
  + (ab)((de)c) 
  + (ac)((de)b) 
  \\
  + \;
  &
  ((a(bc))e)d 
  + ((a(be))c)d 
  + ((a(ce))b)d 
  - ((ab)(ce))d 
  - ((ac)(be))d 
  - ((ae)(bc))d 
  \\
  - \;
  &
  ((a(bd))e)c 
  - ((a(be))d)c 
  - ((a(de))b)c 
  + ((ab)(de))c 
  + ((ad)(be))c 
  + ((ae)(bd))c 
  \\
  - \;
  &
  ((a(ce))b)d 
  - ((ad)b)(ce) 
  + (ab)((ce)d) 
  + (a(ce))(bd) 
  + (ad)(b(ce)) 
  - a(((ce)d)b) 
  \\
  + \;
  &
  ((a(de))b)c 
  + ((ac)b)(de) 
  - (ab)((de)c) 
  - (a(de))(bc) 
  - (ac)(b(de)) 
  + a(((de)c)b) 
  \\
  - \;
  &
  (((ac)b)e)d 
  - (((ae)b)c)d 
  + ((ab)(ce))d 
  + ((ac)(be))d 
  + ((ae)(bc))d 
  - (a((ce)b))d 
  \\
  + \;
  &
  (((ad)b)e)c 
  + (((ae)b)d)c 
  - ((ab)(de))c 
  - ((ad)(be))c 
  - ((ae)(bd))c 
  + (a((de)b))c. 
  \end{align*}  
We straighten each term using right commutativity, and sort the terms 
by association type and then by lex order of the permutation.
Most of the terms cancel, and we obtain an expression identical to \eqref{LTS5part1}.

We finally consider \eqref{LTS3}.
Replacing each occurrence of $\langle a, b, c \rangle$ by the permuted associator $(a,c,b)$ gives
  \begin{align*}
  &
  ((c,e,d),a,b) - ((c,e,d),b,a) - ((c,a,b),e,d) + ((c,b,a),e,d) 
  \\
  &
  - (c,e,(a,d,b)) - (c,(a,e,b),d).
  \end{align*}
Expanding the associators produces the following expression:
  \begin{align*}
  &
  + (((ce)d)a)b - ((c(ed))a)b - ((ce)d)(ab) + (c(ed))(ab)
  \\
  &
  - (((ce)d)b)a + ((c(ed))b)a + ((ce)d)(ba) - (c(ed))(ba)
  \\
  &
  - (((ca)b)e)d + ((c(ab))e)d + ((ca)b)(ed) - (c(ab))(ed)
  \\
  &
  + (((cb)a)e)d - ((c(ba))e)d - ((cb)a)(ed) + (c(ba))(ed)
  \\
  & 
  - (ce)((ad)b) + (ce)(a(db)) + c(e((ad)b)) - c(e(a(db)))
  \\
  &
  - (c((ae)b))d + (c(a(eb)))d + c(((ae)b)d) - c((a(eb))d).
  \end{align*}
We straighten each term using right commutativity, and sort the terms 
by association type and then by lex order of the permutation;
some of the terms cancel:
  \begin{equation}
  \label{LTS6part1}
  \left\{
  \begin{array}{l}
  {}
  - (((ca)b)e)d 
  + (((cb)a)e)d 
  + (((ce)d)a)b 
  - (((ce)d)b)a 
  \\[3pt]
  {}
  - ((c(de))a)b 
  + ((c(de))b)a 
  - (c((ae)b))d 
  + (c((be)a))d 
  \\[3pt]
  {}
  + ((ca)b)(de) 
  - ((cb)a)(de) 
  - (ce)((ad)b) 
  + (ce)((bd)a) 
  \\[3pt]
  {}
  + c(((ad)b)e) 
  + c(((ae)b)d) 
  - c(((bd)a)e)
  - c(((be)a)d)
  \end{array}
  \right. 
  \end{equation}    
Consider the following expression, which clearly vanishes in any Jordan dialgebra:
  \begin{align*}
  &
    c RJ( a, d, e, b ) 
  - c RJ( b, d, e, a ) 
  + RO( ce, a, b, d ) 
  - RO( ce, b, a, d ) 
  \\
  &\quad
  - RO( c, a, de, b ) 
  + RO( c, b, de, a ) 
  + RO( c, a, b, e ) d 
  - RO( c, b, a, e ) d. 
  \end{align*} 
Expanding each term using equations \eqref{RJidentity} and \eqref{ROidentity} gives
  \allowdisplaybreaks
  \begin{align*} 
  &
  c((b(ad))e) + c((b(ae))d) + c((b(de))a) - c((ba)(de)) - c((bd)(ae)) - c((be)(ad))
  \\
  - 
  \;&
  c((a(bd))e) - c((a(be))d) - c((a(de))b) + c((ab)(de)) + c((ad)(be)) + c((ae)(bd))
  \\
  + 
  \;&
  (((ce)b)a)d + (((ce)d)a)b - ((ce)a)(bd) - ((ce)b)(ad) - ((ce)d)(ab) + (ce)((bd)a)
  \\
  - 
  \;&
  (((ce)a)b)d - (((ce)d)b)a + ((ce)b)(ad) + ((ce)a)(bd) + ((ce)d)(ba) - (ce)((ad)b)
  \\
  - 
  \;&
  ((c(de))a)b - ((cb)a)(de) + (ca)((de)b) + (c(de))(ab) + (cb)(a(de)) - c(((de)b)a)
  \\
  + 
  \;&
  ((c(de))b)a + ((ca)b)(de) - (cb)((de)a) - (c(de))(ba) - (ca)(b(de)) + c(((de)a)b)
  \\
  +
  \;&
  (((cb)a)e)d + (((ce)a)b)d - ((ca)(be))d - ((cb)(ae))d - ((ce)(ab))d + (c((be)a))d  
  \\
  -
  \;&
  (((ca)b)e)d - (((ce)b)a)d + ((cb)(ae))d + ((ca)(be))d + ((ce)(ba))d - (c((ae)b))d.
  \end{align*}
We straighten each term using right commutativity, and sort the terms 
by association type and then by lex order of the permutation.
Most of the terms cancel, and we obtain an expression identical to \eqref{LTS6part1}.
\end{proof}

%%%%%%%%%%%%%%%%%%%%%%%%%%%%%%%%%%%%%%%%%%%%%%%%%%%%%%%%%%%%%%%%%%%%%%%%

\section{Universal Leibniz envelopes for Leibniz triple systems} \label{sectionuniversalleibnizenvelopes}

Suppose that $T$ is a Leibniz triple system with product $\langle x, y, z \rangle$.
We construct the free Leibniz algebra $A(T)_L$ on the underlying vector space,
and consider the ideal $I(T) \subseteq A(T)_L$ generated by the elements 
$( x \cdot y ) \cdot z - \langle x, y, z \rangle$,
where $x \cdot y$ is the Leibniz product in $A(T)_L$.
The quotient algebra $U(T) = A(T)_L / I(T)$ is the universal Leibniz envelope of $T$.
If $a, b, c \in T$ then in $U(T)$ we have the identity
$(a \cdot b ) \cdot c \equiv \langle a, b, c \rangle$;
it follows that in $U(T)$ every monomial of degree 3 or more is equal to a monomial of degree 1 or 2.
Furthermore, it is clear from the discussion in Section \ref{sectionfreeleibnizalgebras}
that the intersection of $I(T)$ with $T \oplus T^{\otimes 2}$ is zero.

\begin{theorem} \label{universal envelope}
\emph{(a)}
The universal Leibniz envelope of the Leibniz triple system $T$ is the vector space 
$U(T) = T \oplus ( T \otimes T )$ 
with the Leibniz product
  \begin{alignat*}{2}
  a \cdot b 
  &= 
  ab,
  &\qquad\qquad
  a \cdot bc 
  &= 
  \langle a,b,c \rangle - \langle a,c,b \rangle,
  \\
  ab \cdot c
  &= 
  \langle a,b,c \rangle,
  &\qquad\qquad
  ab \cdot cd 
  &=
  \langle a,b,c \rangle d - \langle a,b,d \rangle c.
  \end{alignat*}
(The simple tensor $a \otimes b$ is denoted $ab$.)

\emph{(b)}
If $T$ has dimension $n$ then $U(T)$ has dimension $n(n{+}1)$.
\end{theorem}

The proof of this theorem is immediate from the universal property of the free Leibniz algebra $A(T)_L$.
However, it is instructive to give a direct proof that the equations of the theorem satisfy the Leibniz
identity, especially in degree 5 where we need the defining identities for Leibniz
triple systems.

\subsection*{Degree 3}

If $a, b, c \in T$ then the definitions immediately give
  \[
  ( a \cdot b ) \cdot c - ( a \cdot c ) \cdot b - a \cdot ( b \cdot c )
  =
  \langle a,b,c \rangle - \langle a,c,b \rangle - \big( \langle a,b,c \rangle - \langle a,c,b \rangle \big)
  =
  0.
  \]
\subsection*{Degree 4}

For $a, b, c, d \in T$ there are three cases for the first term of the Leibniz identity:
  \[
  ( ab \cdot c ) \cdot d,
  \qquad
  ( a \cdot bc ) \cdot d,
  \qquad
  ( a \cdot b ) \cdot cd.
  \]
For the first case, the definitions immediately give
  \[
  ( ab \cdot c ) \cdot d - ( ab \cdot d ) \cdot c - ab \cdot cd
  =
  \langle a,b,c \rangle d - \langle a,b,d \rangle c 
  - ( \langle a,b,c \rangle d - \langle a,b,d \rangle c )
  = 0. 
  \]
For the second case we need to use the equation $\langle a,b,c \rangle \equiv (ab)c$ in $U(T)$:
  \begin{align*}
  &
  ( a \cdot bc ) \cdot d - ( a \cdot d ) \cdot bc - a \cdot ( bc \cdot d )
  \\
  = \;
  &
  ( \langle a,b,c \rangle d - \langle a,c,b \rangle d )
  -
  ( \langle a,d,b \rangle c - \langle a,d,c \rangle b )
  -
  a \langle b,c,d \rangle 
  \\
  = \;
  &
  abcd - acbd - adbc + adcb - a \cdot bcd
  =
  0,
  \end{align*}
using Example \ref{example_LB_Monomials}. The third case is similar.

\subsection*{Degree 5}
If $a, b, c, d, e \in T$ then again there are three cases for the first term:
  \[
  ( ab \cdot cd ) \cdot e,
  \qquad
  ( ab \cdot c ) \cdot de,
  \qquad
  ( a \cdot bc ) \cdot de.
  \]
For the first case, we have
  \begin{align*}
  &
  ( ab \cdot cd ) \cdot e - ( ab \cdot e ) \cdot cd - ab \cdot ( cd \cdot e )
  \\
  = \;
  &
  ( \langle a,b,c \rangle d - \langle a,b,d \rangle c ) \cdot e
  -
  \langle a,b,e \rangle \cdot cd
  -
  ab \cdot \langle c,d,e \rangle
  \\
  = \;
  &
  \langle \langle a, b, c \rangle, d, e \rangle - \langle \langle a, b, d \rangle, c, e \rangle
  -
  \langle \langle a, b, e \rangle, c, d \rangle + \langle \langle a, b, e \rangle, d, c \rangle
  -
  \langle a, b, \langle c, d, e \rangle \rangle,
  \end{align*}
which vanishes by \eqref{LTS-B}. For the second case, we have
  \begin{align*}
  &
  ( ab \cdot c ) \cdot de - ( ab \cdot de ) \cdot c - ab \cdot ( c \cdot de )
  \\
  = \;
  &
  \langle a,b,c \rangle \cdot de 
  - 
  ( \langle a,b,d \rangle e ) - \langle a,b,e \rangle d ) \cdot c 
  -
  ab \cdot ( \langle c,d,e \rangle - \langle c,e,d \rangle )
  \\
  = \;
  &
  \langle \langle a, b, c \rangle, d, e \rangle 
  - 
  \langle \langle a, b, c \rangle, e, d \rangle
  -
  \langle \langle a, b, d \rangle, e, c \rangle 
  + 
  \langle \langle a, b, e \rangle, d, c \rangle
  \\
  &
  -
  \langle a, b, \langle c, d, e \rangle \rangle
  +
  \langle a, b, \langle c, e, d \rangle \rangle 
  = 0
  \end{align*}
by \eqref{bracket of R and L}. For the third case, we have
  \begin{align*}
  &
  ( a \cdot bc ) \cdot de - ( a \cdot de ) \cdot bc - a \cdot ( bc \cdot de )
  \\
  = \;
  &
  ( \langle a,b,c \rangle - \langle a,c,b \rangle ) \cdot de
  -
  ( \langle a,d,e \rangle - \langle a,e,d \rangle ) \cdot bc
  -
  a \cdot ( \langle b,c,d \rangle e - \langle b,c,e \rangle d )
  \\
  = \;
  &
  \langle \langle a, b, c \rangle, d, e \rangle
  -
  \langle \langle a, b, c \rangle, e, d \rangle
  -
  \langle \langle a, c, b \rangle, d, e \rangle
  +
  \langle \langle a, c, b \rangle, e, d \rangle
  \\
  &
  -
  \langle \langle a, d, e \rangle, b, c \rangle
  +
  \langle \langle a, d, e \rangle, c, b \rangle
  +
  \langle \langle a, e, d \rangle, b, c \rangle
  -
  \langle \langle a, e, d \rangle, c, b \rangle
  \\
  &
  -
  \langle a, \langle b, c, d \rangle, e \rangle
  +
  \langle a, e, \langle b, c, d \rangle \rangle
  +
  \langle a, \langle b, c, e \rangle, d \rangle
  -
  \langle a, d, \langle b, c, e \rangle \rangle.
  \end{align*}
We use \eqref{LTS-A} and \eqref{LTS-B} to rewrite the last four terms (with a sign change):
  \begin{align*}
  &
  \langle a, \langle b, c, d \rangle, e \rangle
  -
  \langle a, \langle b, c, e \rangle, d \rangle
  +
  \langle a, d, \langle b, c, e \rangle \rangle
  -
  \langle a, e, \langle b, c, d \rangle \rangle
  \\
  = \;
  &
  \langle \langle a, b, c \rangle, d, e \rangle
  -
  \langle \langle a, c, b \rangle, d, e \rangle
  -
  \langle \langle a, d, b \rangle, c, e \rangle
  +
  \langle \langle a, d, c \rangle, b, e \rangle
  \\
  &
  -
  \langle \langle a, b, c \rangle, e, d \rangle
  +
  \langle \langle a, c, b \rangle, e, d \rangle
  +
  \langle \langle a, e, b \rangle, c, d \rangle
  -
  \langle \langle a, e, c \rangle, b, d \rangle
  \\
  &
  +
  \langle \langle a, d, b \rangle, c, e \rangle
  -
  \langle \langle a, d, c \rangle, b, e \rangle
  -
  \langle \langle a, d, e \rangle, b, c \rangle
  +
  \langle \langle a, d, e \rangle, c, b \rangle
  \\
  &
  -
  \langle \langle a, e, b \rangle, c, d \rangle
  +
  \langle \langle a, e, c \rangle, b, d \rangle
  +
  \langle \langle a, e, d \rangle, b, c \rangle
  -
  \langle \langle a, e, d \rangle, c, b \rangle
  \\
  = \;
  &
  \langle \langle a, b, c \rangle, d, e \rangle
  -
  \langle \langle a, c, b \rangle, d, e \rangle
  -
  \langle \langle a, b, c \rangle, e, d \rangle
  +
  \langle \langle a, c, b \rangle, e, d \rangle
  \\
  &
  -
  \langle \langle a, d, e \rangle, b, c \rangle
  +
  \langle \langle a, d, e \rangle, c, b \rangle
  +
  \langle \langle a, e, d \rangle, b, c \rangle
  -
  \langle \langle a, e, d \rangle, c, b \rangle.
  \end{align*}
This cancels with the first 8 terms, and completes the proof for degree 5. 
  
\subsection*{Degree 6}

If $a, b, c, d, e, f \in T$ then there is only one case:
  \begin{align*}
  &
  ( ab \cdot cd ) \cdot ef 
  -
  ( ab \cdot ef ) \cdot cd 
  -
  ab \cdot ( cd \cdot ef )
  \\
  = \;
  &
  \langle \langle a, b, c \rangle, d, e \rangle f
  -
  \langle \langle a, b, c \rangle, d, f \rangle e
  -
  \langle \langle a, b, d \rangle, c, e \rangle f
  +
  \langle \langle a, b, d \rangle, c, f \rangle e  
  \\
  &
  -
  \langle \langle a, b, e \rangle, f, c \rangle d
  +
  \langle \langle a, b, e \rangle, f, d \rangle c
  +  
  \langle \langle a, b, f \rangle, e, c \rangle d
  -  
  \langle \langle a, b, f \rangle, e, d \rangle c 
  \\
  &
  -
  \langle a, b, \langle c, d, e \rangle \rangle f 
  +
  \langle a, b, f \rangle \langle c, d, e \rangle
  +
  \langle a, b, \langle c, d, f \rangle \rangle e 
  -
  \langle a, b, e \rangle \langle c, d, f \rangle.  
  \end{align*}
Since this is an expression of even degree, we must rewrite it entirely in terms of the binary Leibniz product.
The first eight terms convert directly to this form using the rule 
$\langle \langle a, b, c \rangle, d, e \rangle f = ((((ab)c)d)e)f$.
We rewrite the last four terms (with a sign change) and repeatedly apply the Leibniz identity:
  \begin{align*}
  &
  \langle a, b, \langle c, d, e \rangle \rangle f 
  -
  \langle a, b, \langle c, d, f \rangle \rangle e 
  +
  \langle a, b, e \rangle \langle c, d, f \rangle
  -
  \langle a, b, f \rangle \langle c, d, e \rangle
  \\
  = \;
  & 
  ((ab)((cd)e))f - ((ab)((cd)f))e + ((ab)e)((cd)f) - ((ab)f)((cd)e)
  \\
  = \;
  &
  ((((ab)c)d)e)f - ((((ab)d)c)e)f - ((((ab)e)c)d)f + ((((ab)e)d)c)f
  \\
  &
  - ((((ab)c)d)f)e + ((((ab)d)c)f)e + ((((ab)f)c)d)e - ((((ab)f)d)c)e
  \\
  &
  + ((((ab)e)c)d)f - ((((ab)e)d)c)f - ((((ab)e)f)c)d + ((((ab)e)f)d)c
  \\
  &
  - ((((ab)f)c)d)e + ((((ab)f)d)c)e + ((((ab)f)e)c)d - ((((ab)f)e)d)c
  \\
  = \;
  &
  ((((ab)c)d)e)f - ((((ab)d)c)e)f - ((((ab)c)d)f)e + ((((ab)d)c)f)e
  \\
  &
  - ((((ab)e)f)c)d + ((((ab)e)f)d)c + ((((ab)f)e)c)d - ((((ab)f)e)d)c
  \end{align*}
This cancels with the first 8 terms, and completes the proof.

\begin{theorem}
Over a field of characteristic not 2 or 3, 
every polynomial identity satisfied by the iterated Leibniz bracket $\langle \langle a, b \rangle, c \rangle$
in every Leibniz algebra is a consequence of the defining identities for Leibniz triple systems.
\end{theorem}

\begin{proof}
Suppose to the contrary that there is a polynomial identity $I = I(a_1,\dots,a_n)$ in $n$ indeterminates,
which is satisfied by the iterated Leibniz bracket, but is not a consequence of the defining identities
for Leibniz triple systems.  Such an identity $I$ is a nonzero element of the free Leibniz triple system
$T = T_n$ on $n$ generators $a_1,\dots,a_n$ which is in the kernel of the natural 
evaluation map $\eta\colon T \to L$, 
$\langle a, b, c \rangle \mapsto \langle \langle a, b \rangle, c \rangle$, 
where $L = L_n$ is the free Leibniz algebra on the same $n$ generators.  Let $U(T)$ be the universal Leibniz 
envelope of $T$ as constructed in Theorem \ref{universal envelope}; then there is an injective homomorphism 
of Leibniz triple systems $\iota\colon T \to U(T)^\dagger$ where the dagger denotes the Leibniz triple system 
obtained from the Leibniz algebra $U(T)$ by replacing the binary product $\langle a, b \rangle$ by the 
iterated Leibniz bracket $\langle \langle a, b \rangle, c \rangle$.  By the universal property of the free
Leibniz algebra $L$, there is a (unique) surjective homomorphism $\phi\colon L \to U(T)$, which is the
identity map on the generators $a_1,\dots,a_n$, and which satisfies the condition 
$\phi^\dagger \circ \eta = \iota$, 
where $\phi^\dagger\colon L^\dagger \to U(T)^\dagger$ is the same as $\phi$ but regarded as a homomorphism 
of Leibniz triple systems.  If $\mathrm{ker}(\eta) \ne \{0\}$ then $\eta$ is not injective, and hence
$\iota$ is not injective, which is a contradiction.  This proves that such a polynomial identity $I$
cannot exist.
\end{proof}

%%%%%%%%%%%%%%%%%%%%%%%%%%%%%%%%%%%%%%%%%%%%%%%%%%%%%%%%%%%%%%%%%%%%%%%%

\section{Two-dimensional Leibniz triple systems} \label{sectiontwodimensional}

In this section we give some examples of 2-dimensional Leibniz triple systems, 
and construct their universal Leibniz envelopes. 

\subsection{Leibniz triple systems}

Let $F$ be an algebraically closed field of characteristic 0, 
and let $T$ be a 2-dimensional Leibniz triple system with basis $\{ x, y \}$ 
and product $\langle -,-,- \rangle$. The system $T$ has the following structure constants
where $\alpha_{ijk}, \beta_{ijk} \in F$:
  \begin{alignat*}{2}
  \langle x, x, x \rangle &= \alpha_{111}x + \beta_{111}y, 
  &\qquad\qquad   
  \langle x, x, y \rangle &= \alpha_{112}x + \beta_{112}y,  
  \\
  \langle x, y, x \rangle &= \alpha_{121}x + \beta_{121}y, 
  &\qquad\qquad  
  \langle x, y, y \rangle &= \alpha_{122}x + \beta_{122}y, 
  \\
  \langle y, x, x \rangle &= \alpha_{211}x + \beta_{211}y, 
  &\qquad\qquad 
  \langle y, x, y \rangle &= \alpha_{212}x + \beta_{212}y, 
  \\
  \langle y, y, x \rangle &= \alpha_{221}x + \beta_{221}y, 
  &\qquad\qquad  
  \langle y, y, y \rangle &= \alpha_{222}x + \beta_{222}y.
  \end{alignat*}
Imposing identities \eqref{LTS-A} and \eqref{LTS-B} we obtain a system of quadratic equations
in the indeterminates $\alpha_{ijk}, \beta_{ijk}$.
We present five solutions of these equations: four isolated cases, and a one-parameter family.

By Example \ref{Lie2LeibnizExample}, any Lie triple system is a Leibniz triple system.
Jacobson \cite[page 312]{Jacobson} presents the three isomorphism classes of 2-dimensional Lie triple systems
over an algebraically closed field $F$ of characteristic not two.
Ignoring the system with zero multiplication, we have two cases,
where zero products are omitted:
  \begin{alignat}{4}
  &\langle x, y, x \rangle =  y, 
  &\quad
  &\langle y, x, x \rangle = -y.
  \label{system1}
  \\
  &\langle x, y, x \rangle =  2x, 
  &\quad
  &\langle y, x, x \rangle = -2x, 
  &\quad
  &\langle x, y, y \rangle = -2y, 
  &\quad
  &\langle y, x, y \rangle =  2y.
  \label{system2}
  \end{alignat}
The Leibniz triple systems which are not Lie triple systems are the following:
  \begin{alignat}{2}
  &\langle x, y, y \rangle = x, 
  &\qquad
  &\langle y, y, y \rangle = x.
  \label{system3}
  \\
  &\langle x, y, y \rangle = -x,
  &\qquad
  &\langle y, y, y \rangle =  x.
  \label{system4}
  \\
  &\langle x, y, y \rangle = \zeta x,
  &\qquad
  &\langle y, y, y \rangle = (1-\zeta) x.
  \label{system5}
  \end{alignat}
The parameter $\zeta$ can be any element of the field $F$, including 0.

\begin{conjecture}
Over an algebraically closed field of characteristic 0, 
every 2-dimensional Leibniz triple system is isomorphic to one of the systems \eqref{system1}--\eqref{system5}.
\end{conjecture}
     
Let us verify, for example, that system \eqref{system5} satisfies \eqref{LTS-B}:
  \[
  \langle a, b, \langle c, d, e \rangle \rangle  
  - 
  \langle \langle a, b, c \rangle, d, e \rangle  
  + 
  \langle \langle a, b, d \rangle, c, e \rangle  
  + 
  \langle \langle a, b, e \rangle, c, d \rangle   
  - 
  \langle \langle a, b, e \rangle, d, c \rangle
  \equiv 0.
  \]
We make the following substitutions:
  \[
  a = a_1 x + a_2 y,
  \;\;
  b = b_1 x + b_2 y,
  \;\;  
  c = c_1 x + c_2 y,
  \;\;
  d = d_1 x + d_2 y,
  \;\;
  e = e_1 x + e_2 y.
  \]
For the first term of the identity we note that \eqref{system5} implies that 
$\langle c, d, e \rangle$ is a multiple of $x$, 
and hence that $\langle a, b, \langle c, d, e \rangle \rangle = 0$.
For the second term we obtain
  \begin{align*}
  &
  \langle \; 
  \langle \; a_1 x + a_2 y, \; b_1 x + b_2 y, \; c_1 x + c_2 y \; \rangle, \;
  d_1 x + d_2 y, \; 
  e_1 x + e_2 y \; 
  \rangle 
  \\
  = \;
  & 
  \langle \; 
  a_1 b_2 c_2 \zeta x + a_2 b_2 c_2 (1-\zeta) x, \;
  d_1 x + d_2 y, \; 
  e_1 x + e_2 y \; 
  \rangle 
  \\
  = \;
  & 
  \big( a_1 b_2 c_2 \zeta + a_2 b_2 c_2 (1-\zeta) \big)
  d_2 e_2 \zeta x
  =
  \zeta \big( a_1 \zeta + a_2 (1-\zeta) \big)
  b_2 c_2 d_2 e_2 x.  
  \end{align*}
For the third, fourth and fifth terms we apply appropriate permutations of $c, d, e$
and obtain the same result; hence the alternating sum of these four terms is zero.

\subsection{Universal Leibniz envelopes}

By Theorem \ref{universal envelope} we know that
the universal Leibniz envelope $U(T)$ of the Leibniz triple system $T$ has dimension 6 
and basis $\{ x, y, x^2, xy, yx, y^2 \}$ 
where $x^2$, $xy$, $yx$, $y^2$ denote respectively $x \otimes x$, $x \otimes y$, $y \otimes x$, $y \otimes y$. 
We easily obtain the following structure constants for the universal Leibniz envelopes 
of the 2-dimensional Leibniz triple systems.
We note that for systems \eqref{system1} and \eqref{system2},
even though the Leibniz triple system is a Lie triple system, 
the universal Leibniz envelope is not a Lie algebra, 
since it is not anticommutative.
  \begin{alignat*}{2}
  &\text{System \eqref{system1}:}
  &\qquad
  &
  \begin{array}{l|rrrrrr}
  .   &\quad x   &\quad y   &\quad x^2 & xy   & yx   &\quad y^2 \\
  \midrule
  x   &\quad x^2 &\quad xy  &\quad .   & -y   & y    &\quad . \\ 
  y   &\quad yx  &\quad y^2 &\quad .   & .    & .    &\quad . \\
  x^2 &\quad .   &\quad .   &\quad .   & .    & .    &\quad . \\
  xy  &\quad  y  &\quad .   &\quad .   &  y^2 & -y^2 &\quad . \\
  yx  &\quad -y  &\quad .   &\quad .   & -y^2 &  y^2 &\quad . \\
  y^2 &\quad .   &\quad .   &\quad .   & .    & .    &\quad .
  \end{array}
  \\[8pt]
  &\text{System \eqref{system2}:}
  &\qquad
  &
  \begin{array}{l|rrrrrr}
  .   &\quad x   & y   &\quad x^2 & xy  &\quad yx   &\quad y^2 \\
  \midrule
  x   &\quad x^2 & xy  &\quad .   & -2x  &  2x  &\quad . \\ 
  y   &\quad yx  & y^2 &\quad .   &  2y  & -2y  &\quad . \\
  x^2 &\quad .   & .   &\quad .   & .    & .    &\quad . \\
  xy  &\quad 2x  & -2y &\quad .   &  2(xy{+}yx) & -2(xy{+}yx) &\quad . \\
  yx  &\quad -2x & 2y  &\quad .   & -2(xy{+}yx) &  2(xy{+}yx) &\quad . \\
  y^2 &\quad .   & .   &\quad .   & .    & .    &\quad .
  \end{array}
  \\[8pt]
  &\text{System \eqref{system3}:}
  &\qquad
  &
  \begin{array}{l|rrrrrr}
  .   &\quad x   &\quad y   &\quad x^2 & xy   &\quad yx   &\quad y^2 \\
  \midrule
  x   &\quad x^2 &\quad xy  &\quad .   & .    &\quad .    &\quad . \\ 
  y   &\quad yx  &\quad y^2 &\quad .   & .    &\quad .    &\quad . \\
  x^2 &\quad .   &\quad .   &\quad .   & .    &\quad .    &\quad . \\
  xy  &\quad .   &\quad x   &\quad .   & -x^2 &\quad x^2  &\quad . \\
  yx  &\quad .   &\quad .   &\quad .   & .    &\quad .    &\quad . \\
  y^2 &\quad .   &\quad x   &\quad .   & -x^2 &\quad x^2  &\quad .
  \end{array}  
  \\[8pt]
  &\text{System \eqref{system4}:}
  &\qquad
  &
  \begin{array}{l|rrrrrr}
  .   &\quad x   & y   &\quad x^2 & xy   & yx   &\quad y^2 \\
  \midrule
  x   &\quad x^2 & xy  &\quad .   & .    & .    &\quad . \\ 
  y   &\quad yx  & y^2 &\quad .   & .    & .    &\quad . \\
  x^2 &\quad .   & .   &\quad .   & .    & .    &\quad . \\
  xy  &\quad .   & -x  &\quad .   & x^2  & -x^2 &\quad . \\
  yx  &\quad .   & .   &\quad .   & .    & .    &\quad . \\
  y^2 &\quad .   & x   &\quad .   & -x^2 & x^2  &\quad .
  \end{array}
  \\[8pt]
  &\text{System \eqref{system5}:}
  &\qquad
  &
  \begin{array}{l|rrrrrr}
  .   &\quad x   & y            &\quad x^2 & xy             & yx             &\quad y^2 \\
  \midrule
  x   &\quad x^2 & xy           &\quad .   & .              & .              &\quad . \\ 
  y   &\quad yx  & y^2          &\quad .   & .              & .              &\quad . \\
  x^2 &\quad .   & .            &\quad .   & .              & .              &\quad . \\
  xy  &\quad .   & \zeta x      &\quad .   & -\zeta x^2     & \zeta x^2      &\quad . \\
  yx  &\quad .   & .            &\quad .   & .              & .              &\quad . \\
  y^2 &\quad .   & (1{-}\zeta)x &\quad .   & (\zeta{-}1)x^2 & (1{-}\zeta)x^2 &\quad .
  \end{array}
  \end{alignat*}

%%%%%%%%%%%%%%%%%%%%%%%%%%%%%%%%%%%%%%%%%%%%%%%%%%%%%%%%%%%%%%%%%%%%%%%%

\section*{Acknowledgements}

The first author was supported by a Discovery Grant from NSERC,
the Natural Sciences and Engineering Research Council of Canada.
The second author was supported by the Spanish MEC and Fondos FEDER jointly through project MTM2010-15223, 
and by the Junta de Andaluc\'ia (projects FQM-336 and FQM2467).
She thanks the Department of Mathematics and Statistics at the University of Sas\-kat\-che\-wan
for its hospitality during her visit from March to June 2011.

%%%%%%%%%%%%%%%%%%%%%%%%%%%%%%%%%%%%%%%%%%%%%%%%%%%%%%%%%%%%%%%%%%%%%%%%

\end{document}